\documentclass{amsart}

\usepackage{amssymb}
 \theoremstyle{plain}
\newtheorem{theorem}{Theorem}[section]
\newtheorem*{theorem A}{Theorem A}
\newtheorem*{theorem B}{N\"olker's Theorem}

\theoremstyle{remark}
\newtheorem{remark}{Remark}[section]
\theoremstyle{remark}

\theoremstyle{definition}

\newtheorem{definition}{Definition}[section]

\numberwithin{equation}{section}
\def\({\left ( }
\def\){\right )}
\def\<{\left < }
\def\>{\right >}

\begin{document}

%\noindent {\sc {International Electronic Journal of Geometry}}

%\noindent {\sc \small Volume 1  No. ? pp. 000--000 (2011) \copyright
%IEJG}

\vspace{2cm}

\title{constancy of $\phi-$holomorphic sectional curvature in generalized $g.f.f-$manifolds }

\author{JAE WON LEE and Dae Ho Jin}
\address{Department of Mathematics\\
    Academia Sinica, Taipei 10617, TAIWAN}
\email{jaewon@math.sinica.edu.tw}
\address{Department of Mathematics\\
    Dongguk Unversity, Gyeongju 780-714, South Korea}
\email{jindh@dongguk.ac.kr}

\thanks{The author is supported by ...}


\date{Fabruary 21, 2011 and, in revised form, June 22, 2007.}

%\dedicatory{{\rm (Communicated by H. Hilmi HACISALIHO\v GLU)}}


\begin{abstract}
Tanno \cite{T} provided an algebraic characterization in an almost Hermitian manifold to reduce to a space of constant holomorphic sectional curvature, which he later extended for the Sasakian manifolds as well. In this present paper, we generalize the same characterization in generalized $g.f.f-$manifolds.
\end{abstract}
\maketitle

\section{Introduction}

For an almost Hermitian manifold $(M^{2n}, g, J)$ with $dim(M)=2n > 4$, Tanno \cite{T} has proved;
\begin{theorem}
Let $dim(M) = 2n >4$ and assume that almost Hermitian manifold $(M^{2n}, g, J)$ satisfies
\begin{equation}
R(JX, JY, JZ, JX) =R(X, Y, Z, X)\label{eq:1.1}
\end{equation}
for every tangent vectors $X$, $Y$ and $Z$. Then $(M^{2n}, g, J)$ is of constant holomorphic sectional curvature at $x$ if and only if
\begin{equation}
R(X, JX)X {\textrm{ is proportional to }} JX\label{eq:1.2}
\end{equation}
for every tangent vector $X$ at $x\in M$.
\end{theorem}

Tanno \cite{T} has also proved an analogous theorem for Sasakian manifolds as 

\begin{theorem}
A Sasakian manifold $\geq 5$  is of constant $\phi-$sectional curvature if and only if
\begin{equation}
R(X, \phi X)X {\textrm{ is proportional to } } \phi X\label{eq:1.3}
\end{equation}
for every tangent vector $X$ such that $g(X, \xi)=0$.
\end{theorem}

In this paper, we generalize both Theorem $(1.1)$ and Theorem $(1.2)$  for both an generalized globally framed  $f-$manifold and $\mathcal{S}-$manifold, respectively by proving the followings:

\begin{theorem}
Let $(\bar{M}^{2n+r}, \bar{\phi}, \bar{\xi_{\alpha}},  \bar{\eta}^{\alpha})$, $\alpha\in \{1, \cdots, r\}$, $(n \geq 2)$ be a generalized $g.f.f-$manifold. Then $M^{2n+1}$  is of constant $\phi-$holomorphic sectional curvature if and only if
\begin{equation}
R(X, \phi X)X {\textit{ is proportional to }} \phi X\label{eq:1.4}
\end{equation}
for every vector field $X$ such that $g(X, \bar{\xi_{\alpha}})=0$ for any $\alpha\in \{1, \cdots, r\}$.
\end{theorem}
\smallskip

\begin{theorem}
Let $(\bar{M}^{2n+r}, \bar{\phi}, \bar{\xi_{\alpha}},  \bar{\eta}^{\alpha})$, $\alpha\in \{1, \cdots, r\}$, $(n \geq 2)$ be a $\mathcal{S}-$manifold. Then $M^{2n+1}$  is of constant $\phi-$holomorphic sectional curvature if and only if
\begin{equation}
R(X, \phi X)X {\textit{ is proportional to }} \phi X\label{eq:1.5}
\end{equation}
for every vector field $X$ such that $g(X, \bar{\xi_{\alpha}})=0$ for any $\alpha\in \{1, \cdots, r\}$.
\end{theorem}

\section{ Generalized $g.f.f-$manifolds}

In the class of $f-$structures introduced in $1963$ by Yano \cite{Y}, particularly interesting are the so-called {\it $f-$structures} with complemented frames, also called {\it globally framed $f-$structures} with parallelizable kernel.
\smallskip

A manifold $\bar{M}$ is called a {\it{globally framed f-manifold} ( or $g.f.f$-manifold)} \cite{YK} with a framed $f-$structure  $( \bar{\phi}, \bar{\xi}_{\alpha}, \bar{\eta}^{\alpha}, \bar{g} )$   if it is endowed with a $(1,1)$-tensor field $\bar{\phi}$ of constant rank, such that $ker\bar{\phi}$ is parallelizable i.e. there exist global vector fields $\bar{\xi_{\alpha}}$, $\alpha\in \{1, \cdots, r\}$, with their dual $1$- forms $\bar{\eta}^{\alpha}$, satisfying
\begin{eqnarray}
 \bar{\phi}^2 = -I + \sum^r_{\alpha=1}{\bar{\eta}^{\alpha}\otimes\bar{\xi_{\alpha}}},
\quad
\bar{\eta}^{\alpha}(\bar{\xi_{\beta}})=\delta^{\alpha}_{\beta}, 
\quad \bar{\phi}(\bar{\xi_{\alpha}})=0, 
\quad \bar{\eta}^{\alpha}(\phi)=0     \label{eq:2.1}.
\end{eqnarray}

\begin{equation}
\bar{g}(\bar{\phi}X, \bar{\phi}Y)= \bar{g}(X, Y)-\sum_{\alpha=1}^r\bar{\eta}^{\alpha}(X)\bar{\eta}^{\alpha}(Y)\label{eq:2.2}
\end{equation}
for any $X$, $Y$ $\in \Gamma(T\bar{M})$. Then, for any $\alpha\in \{1, \cdots, r\}$, one has \begin{equation}
\bar{\eta}^{\alpha}(X)=\bar{g}(X, \bar{\xi}_{\alpha}),
\quad  \Phi(X, Y)=\bar{g}(X, \bar{\phi}Y)\label{2.3}
\end{equation}
 for any $X$, $Y$ $\in \Gamma(T\bar{M})$. 
\smallskip

Following the notations in \cite{KN}, we adopt  the curvature tensor $R$, and thus we have $R(X, Y, Z)= \nabla_X \nabla_Y Z - \nabla_Y \nabla_X Z- \nabla_{[X, Y]} Z$, and $R(X, Y, Z, W) = g(R(Z, W, Y), X)$, for any $X$, $Y$, $Z$, $W \in \Gamma(TM)$.

A plane section in $T_p \bar{M}$ is a $\phi-$holomorphic section if there exits a vector $X\in T_p\bar{M}$ orthogonal to $\bar{\xi}_1, \cdots, \bar{\xi}_r$ such that $\{ X, \phi X \}$ span the section. The sectional curvature of a $\phi-$holomorphic secton, denoted by $c(X)=R(X, \phi X, \phi X, X)$, is called a $\phi-$holomorphic sectional curvature.

\begin{definition}(\cite{FP})
A {\it generalized $g.f.f-$space form,} denoted by $M^{2n+r}(F_1, F_2, \mathcal{F})$, is a $g.f.f-$manifold  $(\bar{M}, \bar{\phi}, \bar{\xi}_{\alpha}, \bar{\eta}^{\alpha})$ which admits smooth functions $F_1$, $F_2$, $\mathcal{F}$ such that its curvature tensor field verifies
\begin{eqnarray}
R(X, Y, Z) 
&=& F_1\{ g(\phi X, \phi Z)\phi^2 Y - g(\phi Y, \phi Z)\phi^2 X\} \nonumber\\
&+& F_2\{ g(Z, \phi Y)\phi X - g(Z, \phi X)\phi Y + 2g(X, \phi Y)\phi Z\}\nonumber\\
&+& \sum^{r}_{i, j=1}F_{ij}\{ \bar{\eta}^i(X)\bar{\eta}^j(Z)\phi^2 Y - \bar{\eta}^i(Y)\bar{\eta}^j(Z)\phi^2 X\label{eq:2.4}\\
&+& g(\phi Y, \phi Z)\bar{\eta}^i(X)\bar{\xi}_j - g(\phi X, \phi Z)\bar{\eta}^i(Y)\xi_j\}.\nonumber
\end{eqnarray}
\end{definition}
\smallskip

\begin{theorem}(\cite{FP})
Let  $M^{2n+r}(F_1, F_2, \mathcal{F})$ be  a generalized $g.f.f-$space form. Then, 

\begin{itemize}
\item[(1)] the $\phi-$holomorphic sectional curvature is $c=F_1 + 3F_2$.\\
\item[(2)] $R(\phi X, \phi Y, \phi Z, \phi W) = R(X, Y, Z, W)$ for all $X$, $Y$, $Z$, $W$ $\in TM$.
\end{itemize}
\end{theorem}
\vskip 0.5cm

\subsection*{Proof of Theorem $1.3$}
Let $(\bar{M}^{2n+r}, \bar{\phi}, \bar{\xi_{\alpha}},  \bar{\eta}^{\alpha})$, $\alpha\in \{1, \cdots, r\}$, $(n \geq 2)$ be a generalized $g.f.f-$manifold. To prove the theorem for $n \geq 2$, we shall consider cases when $n=2$ and  when $n > 2$, that is, when $n \geq 3$.

Let $\bar{M}$ be of constant $\phi-$holomorphic sectional curvature. Then (\ref{eq:2.4}) and Theorem $2.1$ give 
\begin{equation}
R(X, \phi X)X=c\phi X \label{eq:2.5}
\end{equation}
 Conversely, let $\{X, Y\}$ be an orthonormal pair of tangent vectors such that 
$\bar{g}(\phi X, Y) = \bar{g}(X, Y) =  \bar{g}(Y, \bar{\xi_{\alpha}}) = 0$, $\alpha\in \{1, \cdots, r\}$ and $n \geq 3$. Then $\dot{X}=\frac{X+Y}{\sqrt{2}}$ and $\dot{Y}=\frac{\phi X - \phi Y}{\sqrt{2}}$  also form an orthonormal pair of tangent vectors such that $\bar{g}(\phi \dot{X}, \dot{Y}) =0$. Then (\ref{eq:2.5}) and curvature properties  give 

\begin{eqnarray}
0  =  R(\dot{X}, \phi \dot{X}, \dot{Y}, \dot{X}) 
    &= &\bar{g}(R(X, \phi X, X), \phi X) - \bar{g}(R(Y, \phi Y, Y), \phi Y)\nonumber\\
	    &-& 2\bar{g}(R(X, \phi Y, Y), \phi Y) + 2\bar{g}(R( X,  \phi X, Y), \phi X)\label{eq:2.6}
\end{eqnarray}
From the assumption, we see that the last two terms of the right hand side vanish. Therefore, we get 
$c(X)=c(Y)$.

Now, if $sp\{ U, V\}$ is $\phi-$holomorphic, then for $\phi U =  aU + bV$ where $a$ and $b$ are constant. Then we have
\begin{equation*}
sp\{U, \phi U\} = sp\{ U, aU + bV\} = sp\{ U, V\}
\end{equation*}
Similarly, 
\begin{equation*}
sp\{ V, \phi V\} = sp\{ U, V\}, \qquad  sp\{U, \phi U\} = sp\{ V, \phi V\}
\end{equation*}
These imply 
\begin{equation*}
R(U, \phi U, U, \phi U) = R(V, \phi V, V, \phi V),  {\textit  or}  \quad c(U) = c(V)
\end{equation*}
If $\sp\{ U, V\}$ is not $\phi-$holomorphic section, then we can choose unit vectors $X \in sp\{U, \phi U\}^{\bot}$ and $Y\in sp\{V, \phi V\}^{\bot}$ such that $sp\{X, Y\}$ is $\phi-$holomorphic. Thus we get
\begin{equation*}
c(U) = c(X) = c(Y) = c(V),
\end{equation*}
which shows that any $\phi-$holomorphic section has the same $\phi-$ holomorphic sectional curvature.

Now, let $n=2$ and let $\{ X, Y\}$  be a set of orthonormal vectors such that $\bar{g}(X, \phi X) = 0$, we have $c(X) = c(Y)$ as before. Using the property (\ref{eq:2.5}), we get 
\begin{eqnarray*}
R(X, \phi X, X)& =& c(X)\phi X\\
R(X, \phi X, Y) &=& R(X, \phi X, Y, \phi Y) \phi Y\\
R(X, \phi Y, X) &= &R( X, \phi Y, X, Y)Y + R(X, \phi Y, X, \phi Y)\phi Y\\
R(X, \phi Y, Y) &= &R( X, \phi Y, Y, X)X + R(X, \phi Y, Y, \phi X) \phi X\\
R(Y, \phi X, Y) &=& R( Y, \phi X, Y, X)X + R(Y, \phi X, Y, \phi X) \phi X\\
R(Y, \phi X, X) &=& R( Y, \phi X, X, Y)Y + R(Y, \phi X, X, \phi Y) \phi Y\\
R(Y, \phi Y, X) &=&  R(Y, \phi Y, X, \phi X) \phi X\\
R(Y, \phi Y, Y) &=& c(Y)\phi Y = c(X)\phi Y
\end{eqnarray*}
Now, define $\ddot{X} = aX+bY$ such that $a^2 + b^2 =1$ and $a^2\neq b^2$. Using above relations, we get 
\begin{equation*}
R(\ddot{X}, \phi \ddot{X}, \ddot{X})= C_1 X + C_2Y + C_3 \phi X + C_4\phi Y
\end{equation*}
Note that $C_1$ and $C_2$ are not necessary for argument and we have
\begin{eqnarray}
C_3 =  a^3 c(X) + ab^2C_5 \label{eq:2.7}\\
C_4 = b^3 c(X) + a^2bC_5,\nonumber
\end{eqnarray}
where $C_5 = R(X, \phi X, Y, \phi Y) + R( X, \phi Y, X,  \phi Y) + R( X, \phi Y, Y, \phi X)$. On the other hand, 
\begin{equation}
R(\ddot{X}, \phi \ddot{X}, \ddot{X})
 = c(\ddot{X})\phi \ddot{X}
= c(\ddot{X})\{ a\phi X + b\phi Y\} \label{eq:2.8}
\end{equation}
Comparing (\ref{eq:2.7}) and (\ref{eq:2.8}), we get
\begin{eqnarray}
a^2c(X) + b^2 C_5 = c(\ddot{X})\label{eq:2.9}\\
b^2c(X) + a^2 C_5 = c(\ddot{X})\label{eq:2.10}
\end{eqnarray}
On Solving (\ref{eq:2.9}) and (\ref{eq:2.10}), we have
\begin{equation*}
c(X) = c(\ddot{X})
\end{equation*}
Similary, we can prove

\begin{equation*}
c(Y) = c(\ddot{Y})
\end{equation*}
Therefore, $\bar{M}$ has constant $\phi-$holomorphic sectional curvature. $\square$

\begin{theorem}\cite{T}
Let $dim(M) = 2n >4$. Then almost Hermitian manifold $(M^{2n}, g, J)$  is of constant holomorphic sectional curvature at $x$ if and only if
\begin{equation}
R(X, JX)X {\textrm{ is proportional to }} JX \label{eq:2.11}
\end{equation}
for every tangent vector $X$ at $x\in M$.
\end{theorem}
\begin{proof}
When $r=0$,  $g.f.f-$manifold is an almost hermitian manifold. Therefore, Theorem $1.3$ imply the condition (\ref{eq:2.11}).
\end{proof}

\begin{remark}
In our case,  the condition $(2)$ in Theorem $2.1$ is satisfied not as same as S. Tanno \cite{T}
has assumed.
\end{remark}

\section{$\mathcal{S}-$space form}
\begin{definition}
A $g.f.f-$manfiold  $(\bar{M}^{2n+r}, \bar{\phi}, \bar{\xi_{\alpha}},  \bar{\eta}^{\alpha})$ is called  an {\it $\mathcal{S}-$manifold} if it is normal and $d\bar{\eta}^{\alpha} = \Phi$, for any  $\alpha\in \{1, \cdots, r\}$, where $\Phi(X, Y) = \bar{g}(X, \phi Y)$ for any $X$, $Y$ $\in \Gamma(T\bar{M})$. The normality condition is expressed by the vanishing of the tensor field $N= N_{\phi} + \sum^{r}_{\alpha=1} 2d\bar{\eta}^{\alpha}\otimes\bar{\xi_{\alpha}}$, $N_{\phi}$ being the Nijenhuis torsion of $\phi$.
\end{definition}
If  $(\bar{M}^{2n+r}, \bar{\phi}, \bar{\xi_{\alpha}},  \bar{\eta}^{\alpha})$ is an $\mathcal{S}-$manifold, then it is known \cite{B1} that
\begin{eqnarray}
(\bar{\nabla}_X\bar{\phi})Y &=&\bar{g}(\bar{\phi}X, \bar{\phi}Y)\bar{\xi} + \bar{\eta}(Y)\bar{\phi}^2(X)\label{eq:3.1}\\
\bar{\nabla}_X\bar{\xi}_{\alpha}&=&-\bar{\phi}X, \label{eq:3.2}
\end{eqnarray}
where $\bar{\xi}=\sum^r_{\alpha=1}\bar{\xi}_{\alpha}$ and $\bar{\eta}=\sum^r_{\alpha=1}\bar{\eta}^{\alpha}$. 

\begin{theorem}\cite{B1}
 An $\mathcal{S}-$manifold $M^{2n+r}$ has constant $\phi-$sectional curvature $c$ if and only if its curvature tensor field satisfies
\begin{eqnarray}
 && 4 \bar{R}(X, Y, Z)\nonumber\\
&& =(c+3r)\{\bar{g}(\bar{\phi}X, \bar{\phi}Z)\bar{\phi}^2Y - \bar{g}(\bar{\phi}Y, \bar{\phi}Z)\bar{\phi}^2X\}\nonumber \\
&&\quad  + (c-r)\{\Phi(Z, Y)\phi X - \Phi(Z, X)\phi Y + 2\Phi(X, Y)\phi Z \} \label{eq:3.3}\\
&& \quad-\bar{g}(\bar{\phi}Z, \bar{\phi}Y) \bar{\eta}(X)\bar{\eta}
+\bar{g}(\bar{\phi}Z, \bar{\phi}X)\bar{\eta}(Y)\bar{\eta}\nonumber\\
&&\quad-\bar{\eta}(Y)\bar{\eta}(Z)\bar{\phi}^2X
+ \bar{\eta}(Z)\bar{\eta}(X)\bar{\phi}^2Y\nonumber
\end{eqnarray}
 for any vector fields $X,  Y,  Z, W \in \Gamma(T\bar{M})$.
\end{theorem}
  
An $\mathcal{S}-$manifold $M^{2n+r}$ with constant $\phi-$sectional curvature $c$ is called an {\it $\mathcal{S}-$space form} and denoted by $M^{2n+r}(c)$.

When $r=1$, an $\mathcal{S}-$space form $M^{2n+1}(c)$ reduces to a Sasakian space form \cite{B2}
and (\ref{eq:3.3}) reduces to 
\begin{eqnarray}
 && 4 \bar{R}(X, Y, Z)\nonumber\\
&& =(c+3)\{\bar{g}(Y, Z)X - \bar{g}(X, Z)Y\}\nonumber \\
&&\quad  + (c-1)\{\Phi(X, Z)\phi Y - \Phi(Y, Z)\phi X + 2\Phi(X, Y)\phi Z  \label{eq:3.4}\\
&& \quad-\bar{g}(Z, Y) \eta(X)\eta + \bar{g}(Z, X)\eta(Y)\eta\nonumber\\
&&\quad- \eta(Y)\eta(Z)X + \eta(Z)\eta(X)Y\}\nonumber
\end{eqnarray}
 for any vector fields $X,  Y,  Z, W \in \Gamma(T\bar{M})$, where $\xi_1= \xi$ and $\eta^1=\eta$.

When $r=0$, an $\mathcal{S}-$space form $M^{2n}(c)$ becomes a complex space form and (\ref{eq:3.3}) moves to 
\begin{eqnarray}
4R(X, Y, Z) &=& c\{ \bar{g}(Y, Z)X - \bar{g}(X, Z)Y\nonumber \\
& &+ \Phi(X, Z)\phi Y - \Phi(Y, Z)\phi X + 2\Phi(X, Y) \phi Z\}\label{eq:3.5}
\end{eqnarray}

\begin{theorem}
Let $(\bar{M}^{2n+r}, \bar{\phi}, \bar{\xi_{\alpha}},  \bar{\eta}^{\alpha})$, $\alpha\in \{1, \cdots, r\}$, $(n \geq 2)$ be a $\mathcal{S}-$manifold. Then $M^{2n+1}$  is of constant $\phi-$holomorphic sectional curvature if and only if
\begin{equation}
R(X, \phi X)X {\textit{ is proportional to }} \phi X\label{eq:3.6}
\end{equation}
for every vector field $X$ such that $g(X, \bar{\xi_{\alpha}})=0$ for any $\alpha\in \{1, \cdots, r\}$.
\end{theorem}
\begin{proof}
An $\mathcal{S}-$space form is a special case of $g.f.f-$ space form, and hence the proof  follows from Theorem $1.3$ and (\ref{eq:3.3}). 
\end{proof}

\begin{theorem} (cf.\;Tanno \cite{T})
A Sasakian manifold $\geq 5$  is of constant $\phi-$sectional curvature if and only if
\begin{equation}
R(X, \phi X)X {\textrm{ is proportional to } } \phi X\label{eq:1.3}
\end{equation}
for every tangent vector $X$ such that $g(X, \xi)=0$.
\end{theorem}
\begin{proof}
When $r=1$, an $\mathcal{S}-$space form $M^{2n+1}(c)$ reduces to a Sasakian space form.
The proof follows from (\ref{eq:3.4}) and Theorem $3.2$.
\end{proof}

\begin{theorem} (cf.\;Kosmanek \cite{K})
A K$\ddot{a}$hlerian manifold $\geq 4$  is of constant $\phi-$sectional curvature if and only if
\begin{equation}
R(X, \phi X)X {\textrm{ is proportional to } } \phi X\label{eq:1.3}
\end{equation}
for every tangent vector $X$  on $M$.
\end{theorem}
\begin{proof}
When $r=0$, an $\mathcal{S}-$space form $M^{2n+1}(c)$ reduces to a K$\ddot{a}$hlerian space form.
The proof follows from (\ref{eq:3.5}) and Theorem $3.2$.
\end{proof}

\end{document}